\documentclass{amsart}

\usepackage{amsmath}
\usepackage{amsfonts}
\usepackage{amssymb}
\usepackage{color}
\usepackage{xcolor}
\usepackage{enumerate}
\usepackage{lineno}
\usepackage{hyperref}

\newtheorem{theorem}{Theorem}[section]
\newtheorem{lemma}[theorem]{Lemma}
\newtheorem{proposition}[theorem]{Proposition}
\newtheorem{notation}[theorem]{Notation}

\theoremstyle{definition}

\newtheorem{claim}[theorem]{Claim}

\theoremstyle{remark}

\numberwithin{equation}{section}

\newtheorem{question}[theorem]{Question}

\newcommand{\N}{\mathbb{N}}

\newcommand{\Bd}{\mathrm{bd}}

\def\int{\mathop{\rm int}\nolimits}
\def\cl{\mathop{\rm cl}\nolimits}

\usepackage{hyperref}
\usepackage{lineno}

\title{The hyperspace of nonblockers of $\mathcal{F}_1(X)$}
\author{Javier Camargo, David Maya and Luis Ortiz }
\thanks{MSC: 54B20, 54F15. Keywords: Continuum, Hyperspace, Non-block hyperspace, Simple closed curve, Non-cut set. \\ The first author thanks La Vicerrector\'ia de Investigaci\'on y Extensi\'on de la Universidad Industrial de Santander y su Programa de Movilidad.}
\begin{document}


\maketitle

\begin{abstract}
A \textsl{continuum} is a compact connected metric space. A non-empty closed subset $B$ of a continuum $X$ \textsl{does not block} $x \in X \setminus B$ provided that the union of all subcontinua of $X$ containing $x$ and contained in $X\setminus B$ is dense in $X$. We denote the collection of all non-empty closed subset $B$ of $X$ such that $B$ does not block each element of $X \setminus B$ by $\mathcal{NB}(\mathcal{F}_1(X))$. In this paper we show some properties of the hyperspace $\mathcal{NB}(\mathcal{F}_1(X))$. Particularly, we prove that the simple closed curve is the unique continuum $X$ such that $\mathcal{NB}(\mathcal{F}_1(X))=\mathcal{F}_1(X)$, given a positive answer to a question posed by Escobedo, Estrada-Obreg\'on and Villanueva in 2012.
\end{abstract}

\section{Introduction}

A \textsl{continuum} is a compact connected metric space. For a continuum $X$, let $2^X$ be the hyperspace of all non-empty closed subset of $X$ and let $\mathcal{F}_1(X)$ be the hyperspace of all one-point subsets of $X$.

Given a continuum $X$ and $A,B\in 2^X$, we say that $B$ does not block $A$ if $A\cap B=\emptyset$ and the union of all subcontinua of $X$ intersecting $A$ and contained in $X\setminus B$ is dense in $X$. Given $\mathcal{L}\subseteq 2^X$, we denote the collection of elements of $2^X$ which do not block each element of $\mathcal{L}$, by $\mathcal{NB}(\mathcal{L})$.



The notions of blocker and nonblocker in hyperspaces have been studied recently by many authors (see \cite{Bobok2016b}, \cite{Bobok2016}, \cite{Escobedo2012}, \cite{Escobedo2017} and \cite{Illanes2011}). In \cite{Illanes2011}, the authors introduce the notion of blocker, and present general properties using different kind of continua. In \cite{Escobedo2012}, the set $\mathcal{NB}(\mathcal{F}_1(X))$ is used to characterizes classes of continua; for instance, it is proved that if $X$ is localy connected, then $\mathcal{NB}(\mathcal{F}_1(X))=\mathcal{F}_1(X)$ if and only if $X$ is the simple closed curve. Naturally, they posed the following question:

\begin{question}\cite[Question 3.3]{Escobedo2012}\label{Question}
	Is a simple closed curve the only continuum $X$ such that $\mathcal{NB}(\mathcal{F}_1(X))=\mathcal{F}_1(X)$?
\end{question}

In \cite{Escobedo2017}, it is defined $\mathcal{NWC}(X), \mathcal{NB}^{\ast}(\mathcal{F}_1(X)), \mathcal{S}(X)$ and $\mathcal{NC}(X)$ (see Notation \ref{espacios}), for some continuum $X$, and it is proved the following result:

\begin{theorem}\cite[Theorem 5.7]{Escobedo2017}\label{Theorem 5.7}
	For a continuum $X$, the following statements are equivalent:
	\begin{enumerate}
		\item $X$ is a simple closed curve,
		\item $\mathcal{NWC}(X)=\mathcal{F}_1(X)$,
		\item $\mathcal{NB}^{\ast}(\mathcal{F}_1(X))=\mathcal{F}_1(X)$,
		\item $\mathcal{S}(X)=\mathcal{F}_1(X)$,
		\item $\mathcal{NC}(X)=\mathcal{F}_1(X)$.		
	\end{enumerate}
\end{theorem}



However, Theorem \ref{Theorem 5.7} does not answer Question \ref{Question}. The authors reformulated Question \ref{Question} as follows:

\begin{question}\cite[Question 4.8]{Escobedo2017}\label{Question1}
	Is there a non-locally connected continuum $X$ such that $\mathcal{NB}(\mathcal{F}_1(X))=\mathcal{F}_1(X)$?
\end{question} 

The main result of this paper is to give a positive answer to Question \ref{Question}. This paper is organized as follows:
In Section 3, we compare $\mathcal{NB}(\mathcal{F}_1(X))$ with others families of closed sets define in \cite{Escobedo2017} (for instance, $\mathcal{NWC}(X)$, $\mathcal{NB}^{\ast}(\mathcal{F}_1(X))$ and $\mathcal{S}(X)$, as in Theorem \ref{Theorem 5.7}). In Section 4, we characterize $\mathcal{NB}(\mathcal{F}_1(X))$ when $X$ is an irreducible decomposable continuum.
Finally, in Section 5, we have the main result in this paper. Theorem \ref{TeoS1} shows that the simple closed curve is the unique continuum $X$ such that the hyperspace of nonblockers of $\mathcal{F}_1(X)$ is in fact, $\mathcal{F}_1(X)$. Thus, we generalize \cite[Theorem~3.2]{Escobedo2012}, give a positive answer to Question \ref{Question}, complete Theorem \ref{Theorem 5.7} and answer Question \ref{Question1} negatively.
%
%
%

\section{Definitions and auxiliary results}

The cardinality of a set $A$ is denoted by $|A|$. The symbol $\mathbb{N}$ represents the set of all positive integers.

Given a subset $A$ of a space $X$, the interior, the closure and the boundary of $A$ are denoted by $\int_X (A)$, $\cl_X (A)$ and $\Bd_X(A)$, respectively. The word \textsl{mapping} stands for a continuous function between topological spaces. A \textsl{simple closed curve} is a space homeomorphic to the unit circle $S^1=\{(x,y)\in\mathbb{R}^2 : x^2+y^2=1\}$.

The \textsl{composant} of a point $z$ in a continuum $X$ is defined by $$\kappa(z)=\bigcup\{L : L \text{ is a proper subcontinuum of }X\text{ and }z\in L\}.$$

A continuum $X$ is called:
\begin{itemize}
\item \textsl{irreducible} provided there are two points $p$ and $q$ of $X$ such that no proper subcontinuum of $X$ contains both $p$ and $q$; we write $X=\mathrm{irr}\{p,q\}$, 
\item \textsl{decomposable} if $X=A\cup B$, where $A$ and $B$ are proper subcontinua of $X$, and 
\item \textsl{indecomposable} if $X$ is not decomposable.
\end{itemize}
Observe that a continuum is irreducible if and only if it has more than one composant. 

Given subsets $U_1, U_2, \ldots, U_l$ of a continuum $X$, we define $$\langle U_1, U_2, \ldots, U_l\rangle= \left\{ A\in 2^X : A\subseteq \bigcup_{i=1}^{l}U_i\text{ and }A\cap U_i\neq\emptyset\text{ for each }i\right\}.$$

The \textsl{Vietoris topology} is the topology on $2^X$ generated by the base consisting of all subsets of the form $\langle U_1, U_2, \ldots, U_l \rangle$, where $U_1, U_2, \ldots, U_l$ are open subsets of the continuum $X$ (see \cite[Proposition~2.1, p.~155]{Michael1951}).

A connected element of $2^X$ is called \textsl{subcontinuum} of the continuum $X$. The hyperspace of all subcontinua of a continuum $X$ is denoted by $\mathcal{C}(X)$. Observe that if $X$ is a continuum, then $\mathcal{F}_1(X) \subseteq \mathcal{C}(X) \subseteq 2^X$. Thus, the hyperspaces $\mathcal{C}(X)$ and $\mathcal{F}_1(X)$ will be considered as subspaces of $2^X$.

For a continuum $X$ and $A,B \in 2^X$ such that $A \subseteq B$, an \textsl{order arc from} $A$ \textsl{to} $B$ is a continuous function $\alpha : [0,1] \to 2^X$ such that $\alpha(0) = A$, $\alpha(1) = B$ and if $0 \leq s < t \leq 1$, then $\alpha(s) \subsetneq \alpha(t)$. When $\alpha([0,1])\subseteq \mathcal{C}(X)$, we say that it is \textsl{an order arc in} $\mathcal{C}(X)$.

The following is \cite[Theorem 14.6]{Nadler}.

\begin{theorem}\label{orderarc}
Let $X$ be a continuum and let $A,B \in \mathcal{C}(X)$ be such that $A \subsetneq B$. Then there exists an order arc in $\mathcal{C}(X)$ from $A$ to $B$.
\end{theorem}

For a continuum $X$, an element $B \in 2^X \setminus \{X\}$:
\begin{itemize}
\item is a \textsl{non-weak cut set} of $X$ provided any two points in $X \setminus B$ belong to a subcontinuum of $X$ contained in $X \setminus B$,
\item \textsl{does not block} $A \in 2^X$ if there exists a mapping $\alpha : [0,1] \to 2^X$ such that $\alpha(0) = A$, $\alpha(1) = X$ and $\alpha(t) \cap B = \emptyset$ for each $t \in [0,1)$.
\item is a \textsl{shore set} of $X$ provided there exists a sequence $(C_{n})_{n \in \mathbb{N}}$ in $\mathcal{C}(X)$ such that $\lim C_n = X$ and each $C_n$ is contained in $X \setminus B$.
\end{itemize}

The following result will be used throughout the current paper.

\begin{proposition}{\cite[Proposition~2.2, p.~3615]{Escobedo2012}}\label{pro:nonb}
For a continuum $X$, $B \in 2^X$ and $x \in X \setminus B$, the following statements are equivalent:
\begin{enumerate}
\item[(a)] $B$ does not block $\{x\}$;
\item[(b)] there exists an order arc $\alpha : [0,1] \to \mathcal{C}(X)$ from $\{x\}$ to $X$ such that $B \cap \alpha(t) = \emptyset$ for each $t \in [0,1)$;
\item[(c)] there exists a sequence $(C_n)_{n \in \mathbb{N}}$ in $\mathcal{C}(X)$ such that $x \in C_n \subseteq C_{n+1} \subseteq X \setminus B$ for each $n \in \mathbb{N}$, and $\bigcup \{ C_n : n \in \mathbb{N}\}$ is dense in $X$; and
\item[(d)] $\bigcup \{ C \in \mathcal{C}(X) : x \in C \subseteq X \setminus B\}$ is dense in $X$.
\end{enumerate}
\end{proposition}

We use the following notation introduced in \cite{Escobedo2017}:
\begin{notation}\label{espacios} 
Given a continuum $X$.
\begin{align*}
\mathcal{NWC}(X) = &\{A\in 2^X : \int_X(A)=\emptyset \text{ and }A\text{ is non-weak cut set of }X\}, \\
\mathcal{NB}(\mathcal{F}_1(X))= &\{B\in 2^X : B \text{ does not block } \{x\} \text{ for each } x \in X \setminus B\}, \\
\mathcal{NB}^{\ast}(\mathcal{F}_1(X))= & \{B\in 2^X : B \text{ does not block } \{x\} \text{ for some } x \in X \setminus B\} \ \text{and} \\
\mathcal{S}(X)=&\{A\in 2^X : A \text{ is shore set of } X\}.
\end{align*}
\end{notation}

In \cite[Theorem 3.2]{Escobedo2017}, it is proved:
\begin{equation}\label{equ00}
\mathcal{NWC}(X)\subseteq \mathcal{NB}(\mathcal{F}_1(X))\subseteq \mathcal{NB}^{\ast}(\mathcal{F}_1(X))\subseteq \mathcal{S}(X).
\end{equation}

The following result is \cite[Theorem 2.7]{Bobok2016b} or \cite[Proposition 6.1]{Escobedo2017}.

\begin{proposition}\label{prop2}
	For each continuum $X$, $X$ is irreducible about $\mathcal{NB}^{\ast}(\mathcal{F}_1(X))$.
\end{proposition}

\section{Hyperspace of non-cut sets}

In \cite{Escobedo2017}, it is studied when the equalities are hold in (\ref{equ00}). In Theorem \ref{Teorema00}, we generalize \cite[Theorem 3.12]{Escobedo2017}.

\begin{lemma}\label{Lema3.3}
	Let $X$ be a continuum such that there exists $p\in X$, where $\{p\}\in\mathcal{S}(X)\setminus \mathcal{NB}^{\ast}(\mathcal{F}_1(X))$. If $X\setminus L$ is disconnected, for some $L\in\mathcal{C}(X)\setminus\{X\}$, then $|\mathcal{NB}^{\ast}(\mathcal{F}_1(X))|=\infty$.
\end{lemma}

\begin{proof}
	Let $L\in\mathcal{C}(X)\setminus\{X\}$ be such that $X\setminus L=U\cup V$, where $U$ and $V$ are non-empty open and disjoint subsets of $X$. Let $p\in X$ such that $\{p\}\in\mathcal{S}(X)\setminus \mathcal{NB}^{\ast}(\mathcal{F}_1(X))$. Hence, there exists a sequence $(C_n)_{n\in\N}$ in $\mathcal{C}(X)$ such that $\lim_{n\to\infty}C_n=X$ and each $C_n$ is contained in $X \setminus \{p\}$. Observe that $X\in \langle X,U,V\rangle$, implies that there exists $k\in\N$ such that $C_m\in\langle X,U,V\rangle$ for each $m\geq k$. Furthermore, $C_m\cap L\neq\emptyset$, by the connectedness of $C_m$, for each $m\geq k$. Since $\{p\}\notin \mathcal{NB}^{\ast}(\mathcal{F}_1(X))$, in light of the fact that (a) and (c) of Proposition~\ref{pro:nonb} are equivalent, we may suppose that $C_i\cap C_j=\emptyset$, for each $i\neq j$. Let $x_k\in C_k\cap U$ and $D_j=(\bigcup_{i=k+1}^{j}C_i)\cup L$, for each $j>k+1$. It is not difficult to see that $x_k\notin D_j, D_{j}\subseteq D_{j+1}$, for each $j>k+1$, and $\lim_{j\to\infty} D_j=X$. Thus, applying again the equivalence between (a) and (c) of Proposition~\ref{pro:nonb}, we obtain that $x_k\in \mathcal{NB}^{\ast}(\mathcal{F}_1(X))$. Therefore, there exists a secuence $(x_i)_{i\geq k}\subseteq \mathcal{NB}^{\ast}(\mathcal{F}_1(X))$, where $x_i\in C_i$ for each $i\geq k$, and $|\mathcal{NB}^{\ast}(\mathcal{F}_1(X))|=\infty$.
\end{proof}

\begin{theorem}\label{Teorema00}
	Let $X$ be a continuum such that $\mathcal{NB}^{\ast}(\mathcal{F}_1(X))$ is finite. Then, $\mathcal{S}(X)$ is finite, and $$\mathcal{NWC}(X)=\mathcal{NB}(\mathcal{F}_1(X))=\mathcal{NB}^{\ast}(\mathcal{F}_1(X))=\mathcal{S}(X).$$
\end{theorem}

\begin{proof}
	Suppose that $\mathcal{S}(X)$ is an infinte set. Hence, there is $p\in X$ such that $\{p\}\in\mathcal{S}(X)\setminus \mathcal{NB}^{\ast}(\mathcal{F}_1(X))$. Furthermore, $X\setminus L$ is connected, for each $L\in\mathcal{C}(X)$, by Lemma \ref{Lema3.3}.
	
	\begin{claim}\label{claim5}
		$\int_X(L)=\emptyset$, for each indecomposable subcontinuum $L$ of $X$.
	\end{claim}

Suppose that there exists an indecomposable subcontinuum $L$ of $X$ such that $\int_X(L)\neq\emptyset$. We know that $X=L\cup \cl_X(X\setminus L)$, where $\cl_X(X\setminus L)$ is a continuum, by Lemma \ref{Lema3.3}. Note that if there is a composant $\kappa$ of $L$ such that $\kappa\cap \cl_X(X\setminus L)=\emptyset$, then $\kappa\subseteq \mathcal{NB}^{\ast}(\mathcal{F}_1(X))$; a contradiction. Thus, each composant of $L$ intersects $\cl_X(X\setminus L)$. Let $z\in L\setminus\cl_X(X\setminus L)$. If $\kappa$ is a composant of $L$ such that $z\notin \kappa$, then $\cl_X(X\setminus L)\cup \kappa$ is a dense subset of $X$ such that does not contain $z$, and for each two points $x,y\in \cl_X(X\setminus L)\cup \kappa$ there is a continuum in $\cl_X(X\setminus L)\cup \kappa$ containing $\{x,y\}$; i.e., $\{z\}\in \mathcal{NB}^{\ast}(\mathcal{F}_1(X))$. Therefore, $\mathcal{NB}^{\ast}(\mathcal{F}_1(X))$ is infinite. A contradiction.

\bigskip

    Let $n=|\mathcal{NB}^{\ast}(\mathcal{F}_1(X))|$. By Lemma \ref{Lema3.3} and Claim \ref{claim5}, there are $Z_1,..., Z_{n+1}$ proper subcontinua of $X$ such that:
    \begin{itemize}
    	\item $X=Z_1\cup ...\cup Z_{n+1}$, and
    	\item $X\setminus (\bigcup_{i\neq j}Z_i)\neq\emptyset$, for each $j\in\{1,...,n+1\}$.
    \end{itemize}
Thus, $\mathcal{NB}^{\ast}(\mathcal{F}_1(X))\subseteq \bigcup_{i\neq j_0}Z_i$, for some $j_0\in\{1,...,n+1\}$. Since $X\setminus Z_{j_0}$ is connected, $\bigcup_{i\neq j_0}Z_i$ is a proper subcontinuum of $X$. We contradict Proposition \ref{prop2}. Therefore, $\mathcal{S}(X)$ is finite. Finally, the theorem follows from \cite[Theorem 3.12]{Escobedo2017}.
\end{proof}

\section{Irreducibility}

The main result in this section is Theorem \ref{Theo3}, where we characterize the hyperspace $\mathcal{NB}(\mathcal{F}_1(X))$, when $X$ is a decomposable irreducible continuum.

\begin{lemma}\label{Lema0}
	Let $X$ be a continuum. Then $\{X\setminus \kappa(x) : x\in X\}\cap 2^X\subseteq \mathcal{NB}(\mathcal{F}_1(X))$.
\end{lemma}

\begin{proof}
	Let $p\in X$ be such that $X\setminus\kappa(p)\in 2^X$. Let $z\in \kappa(p)$. We are going to prove that $X \setminus \kappa(p)$ does not block $\{z\}$. It is well known that $\kappa(p)=\bigcup_{n\in\N}L_n$ such that $L_n\in\mathcal{C}(X)\setminus\{X\}$, $p\in L_n$ and $L_n \subseteq L_{n+1}$ for each $n\in\N$ \cite[Proposition~11.14, p.~203]{Nadler}. Let $k \in\N$ be such that $z\in L_{k}$. Since $\kappa(p)$ is dense, we have that $\cl_X(\{L_m : m \geq k\})=X$. So, by Proposition~®\ref{pro:nonb}, $X\setminus \kappa(p)\in \mathcal{NB}(\mathcal{F}_1(X))$, since $z\in L_m \in\mathcal{C}(X)$ and $(X\setminus\kappa(p))\cap L_m=\emptyset$ for each $m \geq k$. Therefore, $\{X\setminus \kappa(x) : x\in X\}\cap 2^X\subseteq \mathcal{NB}(\mathcal{F}_1(X))$.
\end{proof}

\begin{lemma}\label{Lema2}
	Let $X$ be an irreducible continuum. If $p\in X$ is such that $X\setminus \kappa(p)$ is nonempty compact and $X\setminus\kappa(p)\subseteq \int_X(A)$, for some $A\in\mathcal{C}(X)$, then $A$ is decomposable.
\end{lemma}

\begin{proof}
	Suppose the contrary that $A$ is indecomposable. Note that if $A=X$, then $X\setminus\kappa(p)$ is not compact \cite[Theorem~5, p.~212]{Kuratowski1968II}. Hence, $A\neq X$. Observe that $X\setminus A$ is connected, by \cite[Theorem~3, p.~193]{Kuratowski1968II}. Thus, $X=A\cup B$, where $B=\cl_X(X\setminus A)$. Furthermore, $X\setminus\kappa(p)$ is connected, by \cite[Theorem~3, p.~210]{Kuratowski1968II}. Let $\sigma$ be the composant of $A$ such that $X\setminus\kappa(p)\subseteq \sigma$. We consider two cases: \medskip
	
\noindent \textbf{Case 1.} $\sigma\cap B\neq\emptyset$. \medskip

Since $A$ is indecomposable, it is not difficult to see that there exists $L\in\mathcal{C}(A)$ such that: $(i)\ (X\setminus\kappa(p))\cap L\neq\emptyset$; $ (ii)\ L\cap B\neq\emptyset$; and $(iii)\ \int_A(L)=\emptyset$. Thus, we have that $p \in L\cup B\in\mathcal{C}(X)$ and $(L\cup B)\cap (X\setminus\kappa(p))\neq\emptyset$. Therefore, $X=L\cup B$ and $A\setminus B\subseteq L$. This contradicts $(iii)$. \medskip

\noindent \textbf{Case 2.} $\sigma\cap B=\emptyset$. \medskip

Note that $X$ is irreducible between $p$ and any point of $\sigma$. Therefore, $\sigma\subseteq X\setminus\kappa(p)$ and so $X\setminus\kappa(p)=\sigma$. This contradicts the fact that any composant of $A$ is not compact \cite[Theorem 5, p.212]{Kuratowski1968II}. \medskip

Therefore, $A$ is decomposable.
\end{proof}

\begin{lemma}\label{Lema2.5}
Let $X$ be a decomposable continuum such that $X=\mathrm{irr}\{p,q\}$, for some $p,q\in X$. Then there exist a subcontinuum $L$ of $X$ and non-empty disjoint open connected subsets $U$ and $V$ of $X$ such that $\int_X(L)\neq\emptyset$, 
$X \setminus L = U \cup V$, $X \setminus \kappa(p) \subseteq U$ and $X \setminus \kappa(q) \subseteq V$.
\end{lemma}

\begin{proof}
Let $C$ and $D$ be proper subcontinua of $X$ such that $X=C\cup D$. Since $X=\mathrm{irr}\{p,q\}$, we may suppose that $X\setminus \kappa(q)\subseteq C\setminus D$ and $X\setminus\kappa(p)\subseteq D\setminus C$. Notice that $X\setminus D$ is connected, by \cite[Theorem 3, p.193]{Kuratowski1968II}. Hence, without loss of generality, we may suppose that $C=\cl_X(X\setminus D)$. By Lemma \ref{Lema2}, $C$ is decomposable. Hence, there are proper subcontinua $E$ and $F$ of $C$ such that $C=E\cup F$. Since $X=\mathrm{irr}\{p,q\}$ and $X\setminus\kappa(p)\subseteq D\setminus C$, we have that either $X\setminus\kappa(q)\subseteq E\setminus F$ or $X\setminus\kappa(q)\subseteq F\setminus E$. Suppose that $X\setminus\kappa(q)\subseteq E\setminus F$. Set $L = F$.

Now, observe that if $X\setminus L$ is connected, then $\cl_X(X\setminus L)$ is a proper subcontinuum containing $\{p,q\}$; contradicting the fact that $X=\mathrm{irr}\{p,q\}$. Hence, $X\setminus L$ is disconnected. Furthermore, there exist two disjoint and open connected subsets $U$ and $V$ of $X$ such that $X\setminus L=U\cup V$, $X\setminus \kappa(q)\subseteq U$ and $X\setminus\kappa(p)\subseteq V$, by \cite[Theorem~3, p.~193]{Kuratowski1968II}. Therefore, $L$, $U$ and $V$ satisfy all our requirements.
\end{proof}

\begin{lemma}\label{Lema3}
Let $X$ be a decomposable continuum such that $X=\mathrm{irr}\{p,q\}$, for some $p,q\in X$. Let $A\in 2^X$. Then:
\begin{enumerate}
\item If $A$ does not block $\{x\}$, for any $x\in\kappa(p)\cap\kappa(q)$ such that $x\notin A$, then $A\subseteq \cl_X(X\setminus\kappa(p))\cup\cl_X(X\setminus \kappa(q))$.
\item If $A\subseteq (X\setminus\kappa(p))\cup(X\setminus \kappa(q))$, then $A$ does not block $\{x\}$, for any $x\in\kappa(p)\cap\kappa(q)$.
\end{enumerate} 
\end{lemma}

\begin{proof}
By Lemma \ref{Lema2.5}, there exist a subcontinuum $L$ of $X$ and non-empty disjoint open connected subsets $U$ and $V$ of $X$ such that $\int_X(L)\neq\emptyset$, $X \setminus L = U \cup V$, $X \setminus \kappa(p) \subseteq U$ and $X \setminus \kappa(q) \subseteq V$. Thus, $X=U\cup L\cup V$, where $U, L$ and $V$ are pairwise disjoint connected with non-empty interior subsets of $X$.

\bigskip

We prove \textit{(1)}. Suppose that $A$ does not block $\{x\}$, for any $x\in\kappa(p)\cap\kappa(q)\cap(X\setminus A)$. We show that $A\subseteq \cl_X(X\setminus\kappa(p))\cup\cl_X(X\setminus \kappa(q))$. Suppose the contrary that there exists 
\begin{equation}\label{eq1}
z\in A\cap (X\setminus (\cl_X(X\setminus\kappa(p))\cup\cl_X(X\setminus \kappa(q))))\subseteq A\cap\kappa(p)\cap\kappa(q).
\end{equation} 
Let $H$ be the component of $A$ such that $z\in H$. Note that $\int_X(H)=\emptyset$, by \cite[Proposition 1.1 (b)]{Illanes2011}. Also, observe that if $H\cap \cl_X(U)\neq\emptyset$ and $H\cap \cl_X(V)\neq\emptyset$, then $X=\cl_X(U)\cup H\cup \cl_X(V)$, since $X=\mathrm{irr}\{p,q\}$. It implies that $\int_X(L)\subseteq H$; a contradiction. Thus, we have that either $H\cap \cl_X(V)=\emptyset$ or $H\cap \cl_X(U)=\emptyset$. We only analyze the case when $H\cap \cl_X(V)=\emptyset$. In the following $H\subseteq U\cup L$. We consider two cases: \medskip

\noindent \textbf{Case 1} $H\cap\cl_X(X\setminus\kappa(q))\neq\emptyset$. \medskip 

Let $I=cl_X(X\setminus\kappa(q))\setminus (X\setminus\kappa(q))$. It is possible that $I=\emptyset$. Now we consider two more options: \medskip

\noindent \textbf{Subcase (i)} $I\subseteq H$. \medskip 

Since $X\setminus \kappa(q)$ is a $G_{\delta}$-set (see \cite[Proposition 11.14]{Nadler}), and $H$ is a continuum with empty interior, we have that $\cl_X(X\setminus\kappa(q))\cup H$ is a $G_{\delta}$-set. Since $z\in\kappa(q)$, there exists a proper subcontinuum $R$ of $X$ such that $\{z,q\}\subseteq R$. Hence, $X=\cl_X(X\setminus\kappa(q))\cup H\cup R$, because $X=\mathrm{irr}\{p,q\}$. Thus, $X\setminus R\subseteq \cl_X(X\setminus\kappa(q))\cup H$. Therefore, $X\setminus \kappa(q)$ has non-empty interior, by \cite[Theorem 25.3]{Willard}; but this contradicts \cite[Theorem~2, p.~209]{Kuratowski1968II}. \medskip

\noindent \textbf{Subcase (ii)} $I \setminus H \neq \emptyset$. \medskip

Let $w\in I\setminus H$. Since $w\in \kappa(q)$, there exists a proper subcontinuum $G$ of $X$ such that $\{w,q\}\subseteq G$. Furthermore, since $H$ is a component of $A$, it is not difficult to prove that there exists a subcontinuum $J$ of $G$ such that $w\in J\subseteq X\setminus \{z\}$ and $J\cap \kappa(p)\cap\kappa(q)\cap (X\setminus A)\neq\emptyset$. Let $w'\in J\cap \kappa(p)\cap\kappa(q)\cap (X\setminus A)$. We know that $A$ does not block $\{w'\}$; i.e., there is an order arc $\alpha\colon [0,1]\to \mathcal{C}(X)$ such that $\alpha(0)=\{w'\}, \ \alpha(1)=X$ and $A\cap \alpha(t)=\emptyset$, for each $t<1$. Let $t_0=\min\{t\in [0,1] : \alpha(t)\cap \cl_X(V)\neq\emptyset\}$. Notice that $t_0<1$ and $z\notin \alpha(t_0)$. Furthermore, $\cl_X(X\setminus\kappa(q))$ is a subcontinuum of $X$, by \cite[Theorem 3, p. 210]{Kuratowski1968II}. Thus, by the irreducibility of $X$, $X=\cl_X(X\setminus\kappa(q))\cup J\cup\alpha(t_0)\cup \cl_X(V)$. Therefore, $z\in \cl_X(X\setminus\kappa(q))$ which contradicts (\ref{eq1}). \medskip

\noindent \textbf{Case 2.} $H\cap \cl_X(X\setminus\kappa(q))=\emptyset$. \medskip

Since $H$ is a component of $A$, there exists a proper subcontinuum $E$ of $X$ such that $\cl_X(X\setminus\kappa(q))\subseteq E\subseteq X\setminus H$ and $E\cap\kappa(q)\cap\kappa(p)\cap X\setminus A\neq\emptyset$, by \cite[Theorem 5.4]{Nadler}. Let $y\in E\cap\kappa(q)\cap\kappa(p)\cap X\setminus A$. Since $A$ does not block $\{y\}$, we have an order arc $\sigma\colon [0,1]\to\mathcal{C}(X)$ such that $\alpha(0) = \{y\}$, $\alpha(1) = X$ and $\alpha(t) \cap A = \emptyset$ for each $t \in [0,1)$. So, there exists $s_0\in [0,1)$ such that $\sigma(s_0)$ is a subcontinuum of $X$ intersecting both $E$ and $\cl_X(V)$, and $\sigma(s_0)\cap A=\emptyset$. Since $H\subseteq A$, $H\cap (E\cup \sigma(s_0)\cup \cl_X(V))=\emptyset$; but $E\cup \sigma(s_0)\cup \cl_X(V)$ is a subcontinuum of $X$ such that $\{p,q\}\subseteq E\cup \sigma(s_0)\cup \cl_X(V)$. A contradiction. \medskip

Therefore, we contradict (\ref{eq1}), by (a) and (b), and $A\subseteq \cl_X(X\setminus\kappa(p))\cup\cl_X(X\setminus \kappa(q))$. We complete the first part of the lemma. \medskip

We prove \textit{(2)}. Let $x\in\kappa(p)\cap\kappa(q)$. Let $C$ and $D$ be proper subcontinua of $X$ such that $x\in C\cap D$ and $X=C\cup D$. Without loss of generality, we may suppose that $X\setminus\kappa(q)\subseteq C$ and $X\setminus \kappa(p)\subseteq D$. Let $\alpha_1,\alpha_2\colon [0,1]\to \mathcal{C}(X)$ be an order arc such that $\alpha_1(0)=\alpha_2(0)=\{x\}, \ \alpha_1(1)=C$ and $\alpha_2(1)=D$. Let $t_1=\min\{t\in [0,1] : \alpha_1(t)\cap (A\cup\{p,q\})\neq\emptyset\}$ and $t_2=\min\{t\in [0,1] : \alpha_2(t)\cap (A\cup\{p,q\})\neq\emptyset\}.$ Let $\lambda\colon [0,1]\to \mathcal{C}(X)$ be defined by $$\lambda(t)=\alpha_1(t_1t)\cup\alpha_2(t_2t).$$ Notice that $\lambda(t)\cap A=\emptyset$, for each $t<1$. Furthermore, $\lambda(1)\cap X\setminus\kappa(p)\neq\emptyset$ and $\lambda(1)\cap X\setminus\kappa(q)\neq\emptyset$. Since $X=\mathrm{irr}\{p,q\}$, $\lambda(1)=X$. Therefore, $A$ does not block $\{x\}$, for any $x\in\kappa(p)\cap\kappa(q)$.

\end{proof}

\begin{lemma}\label{Lema4}
Let $X$ be a decomposable continuum such that $X=\mathrm{irr}\{p,q\}$, for some $p,q\in X$. Let $A\in 2^X$. If $A\cap X\setminus\kappa(q)\neq\emptyset$ and there is $w\in (X\setminus\kappa(q))\setminus A$, then $A$ blocks $\{w\}$.
\end{lemma}

\begin{proof}
Let $C$ and $D$ be proper subcontinua such that $X=C\cup D$. Since $X$ is irreducible between $p$ and $q$, we may suppose that $X\setminus\kappa(q)\subseteq C\setminus D$. Let $\alpha\colon [0,1]\to \mathcal{C}(X)$ be an order arc such that $\alpha(0)=\{w\}$ and $\alpha(1)=X$. Since $\int_X(D)\neq\emptyset$, there exists $t_0<1$ such that $\alpha(t_0)\cap D\neq\emptyset$. Observe that $\alpha(t_0)\cap (X\setminus\kappa(q))\neq\emptyset$ and $q\in D$. Hence, $X=\alpha(t_0)\cup D$. Thus, $X\setminus\kappa(q)\subseteq \alpha(t_0)$ and $\alpha(t_0)\cap A\neq\emptyset$. Therefore, $A$ blocks $\{w\}$.
\end{proof}

The next theorem follows from combining Lemmas \ref{Lema0}, \ref{Lema2.5}, \ref{Lema3} and \ref{Lema4}.

\begin{theorem}\label{Theo3}
	Let $X$ be a decomposable continuum such that $X=\mathrm{irr}\{p,q\}$, for some $p,q\in X$. Let $C=X\setminus\kappa(q)$ and $D=X\setminus \kappa(p)$. Then,
	\begin{enumerate}
		\item $\{C,D\}\subseteq 2^X$ if and only if $|\mathcal{NB}(\mathcal{F}_1(X))|=3$. Furthermore, $\mathcal{NB}(\mathcal{F}_1(X))=\{C,D,C\cup D\}$.
		\item $|\{C,D\}\cap 2^X|=1$ if and only if $|\mathcal{NB}(\mathcal{F}_1(X))|=1$. Furthermore, $\mathcal{NB}(\mathcal{F}_1(X))=\{C,D\}\cap 2^X$.
		\item $\{C,D\}\cap 2^X=\emptyset$ if and only if $\mathcal{NB}(\mathcal{F}_1(X))=\emptyset$.
	\end{enumerate}
\end{theorem}

Let $\Sigma_2$ be the dyadic solenoid. It is not difficult to see that $\mathcal{C}(X)\setminus \{X\} \subseteq \mathcal{NB}(\mathcal{F}_1(X))$. It is well known that $\Sigma_2$ is indecomposable and so, irreducible. Also, if $X$ is hereditarily indecomposable, then $\mathcal{NB}(\mathcal{F}_1(X))=\emptyset$ \cite[Remark 3.8]{Escobedo2017}. We have the following natural question:

\begin{question}
Is it possible to characterize the hyperspace $\mathcal{NB}(F_1(X))$ when $X$ is an indecomposable continuum?
\end{question}

A continuum $X$ is of type $\lambda$ provided that $X$ is irreducible and each indecomposable subcontinuum of $X$ has empty interior. By \cite[Theorem 10, p.15]{Thomas}, a continuum $X$ is of type $\lambda$ if and only if admits a finest monotone upper semicontinuous decomposition $\mathcal{G}$ such that each element of $\mathcal{G}$ is nowhere dense and $X/\mathcal{G}$ is an arc.

\begin{theorem}
	Let $X$ be a continuum. If $X$ is of type $\lambda$, then $|\mathcal{NB}(\mathcal{F}_1(X))|=3$.
\end{theorem}

\begin{proof}
	Let $f\colon X\to [0,1]$ be the monotone map such that $\int_X(f^{-1}(t))=\emptyset$, for each $t\in [0,1]$, and  $\{f^{-1}(t) : t\in [0,1]\}$ is the finest monotone upper semicontinuous decomposition of $X$. Observe that $f^{-1}(0)=X\setminus \kappa(q)$ and $f^{-1}(1)=X\setminus\kappa(p)$, by \cite[Theorem 8, p.14]{Thomas}. Therefore, $|\mathcal{NB}(\mathcal{F}_1(X))|=3$, by Theorem \ref{Theo3}.
\end{proof}

\section{A characterization of $S^1$}

\begin{lemma}\label{Lema5}
Let $X$ be a continuum such that $\mathcal{F}_1(X)\subseteq \mathcal{NB}(\mathcal{F}_1(X)$. Suppose that there exists $p\in X$ such that $\{p\}\notin \mathcal{NWC}(X)$. For each $x\in X\setminus\{p\}$, let $$s(x)=\bigcup\{N\in\mathcal{C}(X) : x\in N\subseteq X\setminus \{p\}\}.$$
Then:
	\begin{enumerate}
		\item $s(x)$ is dense, for each $x\in X\setminus\{p\}$.
		\item There are $x,y\in X$ such that $s(x)\cap s(y)=\emptyset$.
		\item $\int_X(s(x))=\emptyset$, for each $x\in X\setminus\{p\}$.
		\item For each $x\in X\setminus \{p\}$, $s(x)=\bigcup\{L_n : n\in\N\}$, where $L_n$ is a subcontinuum of $X\setminus \{p\}$ containing $x$, for each $n\in\N$.
		\item $X\setminus\{p\}=\bigcup_{i\in I}s_i$, where $s_i\cap s_j=\emptyset$, for each $i\neq j$; $s_i=s(x)$ for some $x\in X\setminus\{p\}$, for each $i\in I$; and $I$ is uncountable.
		\item If $K\in\mathcal{C}(X)$ is such that $\int_X(K)\neq\emptyset$, and $A$ is a countable compact subset of $X\setminus K$, then $A$ does not block $\{x\}$, for any $x\in K$.
	\end{enumerate}
\end{lemma}

\begin{proof}
\textit{(1)} follows from the fact that $\{p\}$ does not block $\{x\}$ for each $x \in X \setminus \{p\}$ and Proposition~\ref{pro:nonb}. Since $\{p\}\notin \mathcal{NWC}(X)$, it is clear \textit{(2)}.

It is not difficult to see that if $s(x)\cap s(y)\neq\emptyset$, for some $x,y\in X\setminus\{p\}$, then $s(x)=s(y)$; i.e., the family $\mathcal{P}=\{s(x) : x\in X\setminus\{p\}\}$ is a partition of $X\setminus\{p\}$. By \textit{(2)}, $\mathcal{P}$ is not trivial. Thus, since $s(x)$ is dense for all $x$ (see \textit{(1)}), we have \textit{(3)}.

Let $\{U_n : n\in\N\}$ be a local base of $p$ in $X$. For each $x\in X\setminus\{p\}$, let $L_n$ be the component of $X\setminus U_n$ containing $x$. It is not difficult to see that $s(x)=\bigcup\{L_n : n\in\N\}$. Thus, we prove \text{(4)}. Finally, notice that $s(x)$ is a $F_{\sigma}$-set, for each $x\in X\setminus\{p\}$. Since $X\setminus\{p\}$ is a Baire space (see \cite[Theorem 25.3]{Willard}), we have \textit{(5)}.

Since $\int_X(K)\neq\emptyset$ and $A$ is a countable set, we have that there exists $x\in K$ such that $s(x)\cap A=\emptyset$, by \textit{(1)} and \textit{(5)}. By  \textit{(4)}, there is $(L_n)_{n\in\N}$ a sequence of subcontinua of $X\setminus\{p\}$ such that $x\in L_n$, for each $n\in\N$, such that $s(x)=\bigcup_{n\in\N}L_n$. Since $s(x)$ is dense, it is clear the $\bigcup\{K\cup L_n : n\in\N\}$ is dense in $X$. Furthermore, $A\cap (K\cup L_n)=\emptyset$, for each $n\in\N$. Therefore, $A$ does not block $\{x\}$, by Proposition~\ref{pro:nonb}.
\end{proof}

\begin{proposition}\label{prop3}
	Let $X$ be a continuum such that $\mathcal{NB}(\mathcal{F}_1(X))=\mathcal{F}_1(X)$. Then:
	\begin{enumerate}
		\item $X\setminus K$ is connected, for each $K\in\mathcal{C}(X)$.
		\item $X$ is a decomposable continuum.
		\item $\int_X(K)=\emptyset$ for each indecomposable subcontinuum $K$ of $X$.
	\end{enumerate}
\end{proposition}

\begin{proof}
         Observe that if $\mathcal{F}_1(X)\subseteq \mathcal{NWC}(X)$, then $\mathcal{NWC}(X)=\mathcal{NB}(\mathcal{F}_1(X))$, by (\ref{equ00}). Hence, $X$ is a simple closed curve, by \cite[Theorem 11, p.505]{Bing}, and clearly $S^1$ satisfies \textit{(1), (2)} and \textit{(3)}. Therefore, suppose that there exists $p\in X$ such that $\{p\}\notin\mathcal{NWC}(X)$.         
         
         We prove \textit{(1)}. Let $K\in\mathcal{C}(X)$. Suppose that $X\setminus K$ is disconnected; i.e., there exist open non-empty subsets $U$ and $V$ of $X$ such that $X\setminus K=U\cup V$ and $U\cap V=\emptyset$. It is well known that $K\cup U$ and $K\cup V$ are subcontinua of $X$ \cite[Proposition 6.3]{Nadler}.
         
         By Lemma \ref{Lema5}, there exist $x\in U$ and $y\in V$ such that $s(x)\cap s(y)=\emptyset$. We show that $\{x,y\}\in\mathcal{NB}(\mathcal{F}_1(X))$. Let $w\in X\setminus\{x,y\}$. By Lemma \ref{Lema5} \textit{(5)}, there exists $z\in X\setminus\{p\}$ such that $s(z)\cap (s(x)\cup s(y))=\emptyset$. We denote $s(z)=\bigcup\{L_n : n\in\N\}$, where $L_n$ is a subcontinuum of $X$ such that $z\in L_n\subseteq X\setminus\{p\}$, for each $n\in\N$. Since $s_i$ is dense and connected, it is clear that $s_i\cap K\neq\emptyset$, for each $i\in I$. Thus, without loss of generality, we may suppose that $z\in K$.
         
         We consider two cases: \medskip
         
\noindent \textbf{Case 1.} $w\in s_i\cup K$, for some $i\in I$ where $s_i\cap (s(x)\cup s(y))=\emptyset$. \medskip

Suppose that $w\in s_i$. Since $s_i\cap K\neq\emptyset$, there exists a subcontinuum $T$ of $X$ such that $w\in T\subseteq s_i$ and $T\cap K\neq\emptyset$. Observe that $T\cup K\cup L_n$ is a continuum which does not intersect $\{x,y\}$, for any $n\in\N$. Since $s(z)\subseteq \bigcup\{T\cup K\cup L_n : n\in\N\}$ and $s(z) $ is dense, then $\cl_X(\bigcup\{T\cup K\cup L_n : n\in\N\})=X$. Therefore, $\{x,y\}$ does not block $\{w\}$, by Proposition~\ref{pro:nonb}. If $w\in K$, then $\bigcup\{K\cup L_n : n\in\N\}$ is dense, and $\{x,y\}$ does not block $\{w\}$, by Proposition~\ref{pro:nonb}. \medskip
                               
\noindent \textbf{Case 2.} $w\in (s(x)\cup s(y))\setminus K$. \medskip

Suppose that $w\in s(x)\cap U$. Since $\{x\}\in\mathcal{NB}(\mathcal{F}_1(X))$, there is an order arc $\alpha\colon [0,1]\to \mathcal{C}(X)$ such that $\alpha(0)=\{w\}, \alpha(1)=X$ and $x\notin\alpha(t)$ for each $t<1$. Let $t_0=\min\{t\in [0,1] : \alpha(t)\cap (K\cup V)\neq\emptyset\}$. Since $K\cup V$ is a continuum with non-empty interior, $t_0<1$ and $x\notin \alpha(t_0)$. Furthermore, it is not difficult to see that $\alpha(t_0)\cap V=\emptyset$. Hence, $y\notin \alpha(t_0)$. Thus, $\alpha(t_0)\cup K\cup L_n$ is a subcontinuum of $X$ such that it does not intersect $\{x,y\}$, for each $n\in\N$. Since $s(z)\subseteq \bigcup\{\alpha(t_0)\cup K\cup L_n : n\in\N\}$, we have that $\cl_X(\bigcup\{\alpha(t_0)\cup K\cup L_n : n\in\N\})=X$. Therefore, $\{x,y\}$ does not block $\{w\}$, by Proposition~\ref{pro:nonb}. \medskip

We have that $\{x,y\}\in\mathcal{NB}(\mathcal{F}_1(X))$. A contradiction. This completes the proof of \textit{(1)}. \medskip

We prove \textit{(2)}. Suppose that $X$ is indecomposable. Let $p,q\in X$ such that $\kappa(p)\neq\kappa(q)$. Let $x\in X\setminus\{p,q\}$. Note that if $x\notin \kappa(p)\cup\kappa(q)$, then $\kappa(x)$ is dense and $\kappa(x)\cap\{p,q\}=\emptyset$. Hence, $\{p,q\}$ does not block $\{x\}$. Thus, $x\in \kappa(p)\cup\kappa(q)$. Suppose that $x\in\kappa(p)$. Since $\{p\}\in\mathcal{NB}(\mathcal{F}_1(X))$, there is an order arc $\alpha\colon [0,1]\to \mathcal{C}(X)$ such that $\alpha(0)=\{x\}, \alpha(1)=X$ and $p\notin \alpha(t)$, for each $t<1$. Observe that $\alpha(t)\subseteq \kappa(p)$, for each $t<1$. Thus, $\alpha(t)\cap\{p,q\}=\emptyset$ for each $t<1$, and $\{p,q\}$ does not block $\{x\}$. Thus, $\{p,q\}\in\mathcal{NB}(\mathcal{F}_1(X))$. A contradiction. Therefore, $X$ is a decomposable continuum. \medskip

We prove \textit{(3)}. Suppose the contrary that there exists a subcontinuum $K$ of $X$ such that $K$ is indecomposable and $\int_X(K)\neq\emptyset$. By \textit{(2)}, $X\setminus K\neq\emptyset$. Furthermore, $W=\cl_X(X\setminus K)$ is a continuum, by \textit{(1)}. For each $x\in K$, we denote the composant of $x$ in $K$ by $\kappa_K(x)$. Observe that if $\kappa_K(z)\cap W=\emptyset$, for some $z\in K$, and $x\in K\setminus \kappa_K(z)$, then $\{x\}$ blocks $\{z\}$. This contradicts that $\mathcal{NB}(\mathcal{F}_1(X))=\mathcal{F}_1(X)$. Thus, $\kappa_K(x)\cap W\neq\emptyset$, for each $x\in K$.

Let $x$ and $y$ be points of $\int_X(K)$ such that $\kappa_K(x)\cap\kappa_K(y)=\emptyset$. We see that $\{x,y\}\in\mathcal{NB}(\mathcal{F}_1(X))$. Let $z\in X\setminus \{x,y\}$. Note that if $z\notin \kappa_K(x)\cup\kappa_K(y)$, then there is a subcontinuum $H$ of $X$ such that $\{z\}\cup W\subseteq H$ and $H\cap\{x,y\}=\emptyset$. Since $\int_X(H)\neq\emptyset$, $\{x,y\}$ does not block $\{z\}$, by Lemma \ref{Lema5} \textit{(6)}. Thus, $z\in \kappa_K(x)\cup\kappa_K(y)$. 
Suppose that $z\in \kappa_K(x)$. Since $\mathcal{NB}(\mathcal{F}_1(X))=\mathcal{F}_1(X)$, there exists an order arc $\alpha\colon [0,1]\to \mathcal{C}(X)$, such that $\alpha(0)=\{z\}, \alpha(1)=X$ and $x\notin \alpha(t)$ for each $t<1$. Let $$t_0=\min\{t\in [0,1] : \alpha(t)\cap W\neq\emptyset\}.$$
Observe that $t_0<1$ and $\alpha(t_0)\subseteq K$. Hence, $\alpha(t_0)\subseteq \kappa_K(x)$ and $y\notin \alpha(t_0)$. Thus, $\alpha(t_0)\cup W$ is a subcontinuum of $X$ such that $\int_X(\alpha(t_0)\cup W)\neq\emptyset$ and $\{x,y\}\cap (\alpha(t_0)\cup W)=\emptyset$. Therefore, $\{x,y\}$ does not block $\{z\}$, by Lemma \ref{Lema5} \textit{(6)}. This contradicts that $\mathcal{NB}(\mathcal{F}_1(X))=\mathcal{F}_1(X)$ and $\int_X(K)=\emptyset$ for each indecomposable subcontinuum $K$ of $X$.
\end{proof}

\begin{theorem}\label{TeoS1}
Let $X$ be a continuum. Then $X$ is a simple closed curve if and only if $\mathcal{NB}(\mathcal{F}_1(X))=\mathcal{F}_1(X)$.
\end{theorem}

\begin{proof}
It is clear that if $X$ is a simple closed curve, then $\mathcal{NB}(\mathcal{F}_1(X))=\mathcal{F}_1(X)$. 
Suppose that $\mathcal{NB}(\mathcal{F}_1(X))=\mathcal{F}_1(X)$ and we prove that $X$ is a simple closed curve.
By \cite[Theorem~11, p.~505]{Bing}, we need to show that $\mathcal{F}_1(X)\subseteq \mathcal{NWC}(X)$. Suppose the contrary that there exists $p\in X$ such that $\{p\}\notin \mathcal{NWC}(X)$.	
By Proposition \ref{prop3}, $X=A_1\cup A_2\cup A_3$, where $A_i\in\mathcal{C}(X)$ for $i\in\{1,2,3\}$, and $X\setminus (A_i\cup A_j)\neq\emptyset$, for each $i,j\in\{1,2,3\},\ i\neq j$. Note that $p\in A_1\cap A_2\cap A_3$, by Lemma \ref{Lema5} \textit{(3)}. We will adopt the notation of Lemma~\ref{Lema5}.

Given $q \in X$ and $B \in 2^X$, define $$\Lambda(q,B) = \bigcup\{ K \in \mathcal{C}(X) : q \in K \subseteq X \setminus (B \cup\{p\})\}.$$ We shall prove the following claim.

\begin{claim}\label{claimN}
If $\{k,m,n\} = \{1,2,3\}$, $x \in X \setminus (A_m \cup A_n)$, $y \in X \setminus (A_k \cup A_n \cup s(x))$ and $B = \{x,y\}$, then there exists $q \in (s(x) \cup s(y)) \cap (X \setminus A_n)$ such that $$\Lambda(q,B) \cap A_n = \emptyset, \Lambda(q,B) \cap (X \setminus (A_m \cup A_n)) \neq \emptyset \ \text{and} \ \Lambda(q,B) \cap (X \setminus (A_k \cup A_n)) \neq \emptyset.$$
\end{claim}

Since $\mathcal{NB}(\mathcal{F}_1(X))=\mathcal{F}_1(X)$, there is $q \in X\setminus \{x,y\}$ such that $\{x,y\}$ blocks $\{q\}$.

Suppose that $q\in A_n \cup (X \setminus (s(x) \cup s(y)))$. Notice that $\int_X(A_n)\neq\emptyset$. Hence, using Lemma \ref{Lema5}, it is not difficult to see that there exists a subcontinuum $L$ of $X$ such that $\{q\}\cup A_n\subseteq L$, and $L\cap \{x,y\}=\emptyset$.  Thus, $\{x,y\}$ does not block $\{q\}$, by Lemma \ref{Lema5} \textit{(6)}. A contradiction. Therefore, 
\begin{equation}\label{claim6}
q \in (s(x)\cup s(y))\cap (X\setminus A_n).
\end{equation}

Without loss of generality, we may suppose that $q\in s(x)\cap (X\setminus A_n)$. Note that if $\Lambda(q,B) \cap A_n\neq\emptyset$, then there is $K\in\mathcal{C}(X)$ such that $q \in K$, $K\cup A_n\in\mathcal{C}(X)$ and $\{x,y\}\cap (K\cup A_n)=\emptyset$. Hence, $\{x,y\}$ does not block $\{q\}$, by Lemma \ref{Lema5} \textit{(6)}. A contradiction. Thus, $\Lambda(q,B)\cap A_n=\emptyset$.

Since $\{x,q\}\subseteq s(x)$, there exists a subcontinuum $N$ of $s(x)$ such that $N=\mathrm{irr}\{x,q\}$ \cite[Theorem 1, p.192]{Kuratowski1968II}. Hence, it is not difficult to see that there exists a proper subcontinuum $M$ of $N$, such that $q\in M\subseteq N\setminus \{x\}$ and $M\cap (X\setminus (A_m\cup A_n))\neq\emptyset$. Thus, $M\subseteq \Lambda(q,B)$ and $\Lambda(q,B)\cap (X\setminus (A_m\cup A_n))\neq\emptyset$.

We know that $\{x\}\in\mathcal{NB}(\mathcal{F}_1(X))$. Hence, there is an order arc $\alpha\colon [0,1]\to \mathcal{C}(X)$ such that $\alpha(0)=\{q\}, \alpha(1)=X$ and $x\notin \alpha(t)$, for each $t<1$. Let $t_0=\min\{t\in [0,1] : \alpha(t)\cap A_n\neq\emptyset\}$. Since $\int_X(A_n)\neq\emptyset$,  $t_0<1$. Furthermore, since $\alpha(t)\cap A_n=\emptyset$ for each $t<t_0$, $\alpha(t)\subseteq \Lambda(q,B)$.

If $y \notin \alpha(t_0)$, then $\{x,y\}$ does not block $\{q\}$, by Lemma \ref{Lema5} (6). A contradiction. Thus, $y \in \alpha(t_0)$. Since $y \in X\setminus (A_k\cup A_n)$, there is $t<t_0$ such that $\alpha(t)\cap X\setminus (A_k\cup A_n)\neq\emptyset$ and $y \notin \alpha(t)$. Thus, $\Lambda(q,B) \cap (X\setminus (A_k\cup A_n))\neq\emptyset$. This ends the proof of Claim \ref{claimN}. \medskip

Let $x_0\in s_i\cap (X\setminus (A_2\cup A_3))$ and let $y_0\in s_j\cap (X\setminus (A_1\cup A_3))$, for some $i,j\in I, i\neq j$ (see Lemma \ref{Lema5} \textit{(1)}). Let $B_0 = \{x_0,y_0\}$. By Claim~\ref{claimN}, there exists $w \in (s_i \cup s_j) \cap (X \setminus A_3)$ such that $$\Lambda(w,B_0)\cap A_3=\emptyset,  \Lambda(w,B_0) \cap (X\setminus (A_2\cup A_3))\neq\emptyset,\ \text{and} \ \Lambda(w,B_0)\cap (X\setminus (A_1\cup A_3))\neq\emptyset. $$

Let $y_1\in s_j\cap (X\setminus (A_1\cup A_2))$. Set $B_1 = \{x_0,y_1\}$. Claim~\ref{claimN} guarantees that there exists $z \in (s_i \cup s_j) \cap (X \setminus A_2)$ such that $$ \Lambda(z,B_1)\cap A_2=\emptyset, \ \Lambda(z,B_1)\cap (X\setminus (A_2\cup A_3))\neq\emptyset,\ \text{and}\ \Lambda(z,B_1)\cap (X\setminus (A_1\cup A_2))\neq\emptyset.$$

For sake of simplicity, $L_w$ and $L_z$ will denote $\Lambda(w,B_0)$ and $\Lambda(z,B_1)$, respectively. Thus, 
\begin{equation}\label{eq2}
L_w\cap A_3=\emptyset,  L_w \cap (X\setminus (A_2\cup A_3))\neq\emptyset,\ \text{and} \ L_w\cap (X\setminus (A_1\cup A_3))\neq\emptyset.
\end{equation}
and
\begin{equation}\label{eq3}
L_z\cap A_2=\emptyset, \ L_z\cap (X\setminus (A_2\cup A_3))\neq\emptyset,\ \text{and}\ L_z\cap (X\setminus (A_1\cup A_2))\neq\emptyset.
\end{equation}

Assume first that $z \in s_i \cap (X \setminus A_2)$. \medskip

We know that $p\in A_2\cap A_3$. Furthermore, note that if $A_2\cap A_3$ is connected, then $A_1\cup (A_2\cap A_3)$ is a proper subcontinuum of $X$ which contradicts Proposition \ref{prop3} \textit{(1)}. Thus, $A_2\cap A_3$ is disconnected. Let $P$ be the component of $A_2\cap A_3$ such that $p\in P$. 
Let $E$ be a component of $A_2\cap A_3$ such that $E\neq P$. Observe that $E\subseteq X\setminus\{p\}$. Let $l\in I$ be such that $E\subseteq s_l$. Since $s_l$ is dense, there exists a subcontinuum $H$ of $X$ such that $E\subseteq H\subseteq s_l$ and $H\cap A_1\neq\emptyset$. Let $\alpha\colon [0,1]\to \mathcal{C}(X)$ be an order arc such that $\alpha(0)=E$ and $\alpha(1)=H$ (see Theorem \ref{orderarc}). Let $t_0=\min\{t\in [0,1] : \alpha(t)\cap A_1\neq\emptyset\}$. Let $G_E=\alpha(t_0)$. It is clear that $G_E\subseteq s_l\cap (A_2\cup A_3)$ and $G_E\cap A_1\neq\emptyset$. Note that $\{x_0,p\}\cap G_E=\emptyset$. Hence, if $G_E\cap L_w\neq\emptyset$, then $G_E\subseteq L_w$ and $L_w\cap A_3\neq\emptyset$. This contradicts (\ref{eq2}). Similarly, if $G_E\cap L_z\neq\emptyset$, then $G_E\subseteq L_z$ and $L_z\cap A_2\neq\emptyset$, contradicting (\ref{eq3}). Thus, $G_E\cap (L_w\cup L_z)=\emptyset$. 

Let $W=A_1\cup P\cup (\bigcup\{G_E : E \text{ is component of }A_2\cap A_3\text{ and }p\notin E\})$. Let $x_w\in L_w\cap (X\setminus (A_1\cup A_3))$ and $x_z\in L_z\cap (X\setminus (A_1\cup A_2))$ (see (\ref{eq2}) and (\ref{eq3})). Observe that $\{x_w,x_z\}\cap W=\emptyset$.

Let $x'\in X\setminus\{x_w,x_z\}$. If $x'\in W$, then there exists a subcontinuum with nonempty interior $Q$ of $X$ such that $x'\in Q$ and $\{x_z,x_w\}\cap Q=\emptyset$. Hence, $\{x_w,x_z\}$ does not block $\{x'\}$, by Lemma \ref{Lema5} \textit{(6)}. Thus, $x'\in (X\setminus (A_1\cup A_2))\setminus W$ or $x'\in (X\setminus (A_1\cup A_3))\setminus W$. Suppose that $x'\in (X\setminus (A_1\cup A_2))\setminus W$. Since $\{x_z\}\in\mathcal{NB}(\mathcal{F}_1(X))$, there exists an order arc $\alpha\colon [0,1]\to \mathcal{C}(X)$ such that $\alpha(0)=\{x'\}, \alpha(1)=X$ and $x_z\notin \alpha(t)$ for each $t<1$. Let $t_0=\min\{t\in [0,1] : \alpha(t)\cap (A_1\cup A_2)\neq\emptyset\}$. Since $x_w\in X\setminus (A_1\cup A_3)\subseteq\int_X(A_2)$, $x_w\notin \alpha(t_0)$. Hence, $\alpha(t_0)\cap \{x_z,x_w\}=\emptyset$. Furthermore, $A_1\cup (A_2\cap A_3)\subseteq W$. Hence, $\alpha(t_0)\cap W\neq\emptyset$. Thus, there exists a subcontinuum $H$ of $X$ such that $\{x'\}\cup A_1\subseteq H$ and $H\cap\{x_z,x_w\}=\emptyset$. Therefore, $\{x_w,x_z\}$ does not block $\{x'\}$, by Lemma \ref{Lema5} \textit{(6)}. Similarly we have that $\{x_w,x_z\}$ does not block $\{x'\}$, if $x'\in (X\setminus (A_1\cup A_3))\setminus W$. Therefore, $\{x_z,x_w\}\in\mathcal{NB}(\mathcal{F}_1(X))$. A contradiction.

\vspace{0.2cm}

We have just to consider that $z\in s_j\cap (X\setminus A_2)$. \medskip

Note that $L_z\subseteq s_j$. Let $P$ be the component of $A_2\cap A_3$ such that $p\in P$. Let $\mathcal{E}$ be the collection of all the components of $A_2\cap A_3$ containing no the point $p$. For each $E\in\mathcal{E}$, $E\subseteq s_l$ for some $l\in I$. If $l\neq j$, let $G_E\in \mathcal{C}(X)$ such that $E\subseteq G_E\subseteq s_l\cap (A_2\cup A_3)$ and $G_E\cap A_1\neq\emptyset$. Note that  $G_E\cap L_z=\emptyset$ (see (\ref{eq3})). Let $\mathcal{D}=\{F\in\mathcal{E} : F\subseteq s_j\}.$
For each $F\in\mathcal{D}$, let $$J_{F}=\bigcup\{N\in\mathcal{C}(X) : F\subseteq N\subseteq X\setminus \{y_1,p\}\}.$$ 
It is clear that $J_F\cap L_z=\emptyset$, by (\ref{eq3}). Furthermore, $G_E\cap L_w=\emptyset$, for each $E\in \mathcal{E}\setminus\mathcal{D}$, and $J_F\cap L_w=\emptyset$, for each $F\in\mathcal{D}$, by (\ref{eq2}). We consider two cases: \medskip

\noindent \textbf{Case 1.} $J_F\cap A_1\neq\emptyset$, for each $F\in \mathcal{D}$. \medskip

Let $W=A_1 \cup P\cup (\cup\{J_F : F\in\mathcal{D}\}) \cup (\cup\{G_E : E\in\mathcal{E}\setminus\mathcal{D}\})$.
Let $y_w\in L_w\cap (X\setminus (A_1\cup A_3))$ and $y_z\in L_z\cap (X\setminus (A_1\cup A_2))$ (see (\ref{eq2}) and (\ref{eq3})). Let $x\in X\setminus \{y_w,y_z\}$. If $x\in W$, there exists a subcontinuum $H$ of $W$ such that $A_1\cup\{x\}\subseteq H$. Thus, $\{y_w,y_z\}$ does not block $\{x\}$, by Lemma \ref{Lema5} \textit{(6)}. If $x\notin W$, then $x\in (X\setminus (A_1\cup A_3\cup W))\cup (X\setminus (A_1\cup A_2\cup W))$. Suppose that $x\in X\setminus (A_1\cup A_3\cup W)$. Since $\{y_w\}\in\mathcal{NB}(\mathcal{F}_1(X))$, there is an order arc $\alpha\colon [0,1]\to\mathcal{C}(X)$ such that $\alpha(0)=\{x\}, \alpha(1)=X$ and $y_w\notin\alpha(t)$, for each $t<1$. Let $t_0=\min\{t\in [0,1] : \alpha(t)\cap (A_1\cup A_3)\neq\emptyset\}$. Since $\int_X(A_1\cup A_3)\neq\emptyset$, $t_0<1$. Furthermore, $\alpha(t_0)\subseteq A_2$. Hence, $\alpha(t_0)\cap\{y_w,y_z\}=\emptyset$. Observe that $\mathrm{bd}_X(A_2)\subseteq W$. Thus, $\alpha(t_0)\cap W\neq\emptyset$. Therefore, there exists a subcontinuum $H$ of $X$ such that $\alpha(t_0)\cup A_1\subseteq H$ and $H\cap\{y_w,y_z\}=\emptyset$, and $\{y_w,y_z\}$ does not block $\{x\}$, by Lemma \ref{Lema5} \textit{(6)}. Similarly if $x\in X\setminus (A_1\cup A_2\cup W)$. We have that $\{y_w,y_z\}\in\mathcal{NB}(\mathcal{F}_1(X))$. A contradiction. \medskip

\noindent \textbf{Case 2.} $J_F\cap A_1=\emptyset$, for some $F\in\mathcal{D}$. \medskip 

Since $F\cap A_1=\emptyset$, there are open subsets $B_2$ and $B_3$, of $A_2$ and $A_3$, respectibily, such that $F\subseteq B_2\cap B_3$ and $(\cl_X(B_2)\cup\cl_X(B_3))\cap (A_1\cup\{y_1\})=\emptyset$. Let $Y_2$ and $Y_3$ be the components of $\cl_X(B_2)$ and $\cl_X(B_3)$, respectibily, such that $F\subseteq Y_2\cap Y_3$. By \cite[Corollary 5.5]{Nadler}, $Y_2\cap (X\setminus (A_3\cup A_1))\neq\emptyset$ and $Y_3\cap (X\setminus (A_1\cup A_2))\neq\emptyset$. Thus, $Y_2\cup Y_3\subseteq J_F$ and: 
\begin{equation}\label{eq5}
J_F\cap A_1=\emptyset, \ J_F\cap (X\setminus (A_1\cup A_2))\neq\emptyset,\ \text{ and }\ J_F\cap (X\setminus (A_1\cup A_3))\neq\emptyset.
\end{equation}

Let $P'$ be the component of $A_1\cap A_2$ such that $p\in P'$. Let $Q$ be a component of $A_1\cap A_2$ such that $p\notin Q$. Then $Q\subseteq s_u$, for some $u\in I$. Since $s_u$ is dense, there is $H_Q\in\mathcal{C}(X)$ such that $Q\subseteq H_Q\subseteq (A_1\cup A_2)\cap s_u$ and $H_Q\cap A_3\neq\emptyset$. Note that $H_Q\cap\{y_1,p\}=\emptyset$. Hence, if $H_Q\cap (L_z\cup J_F)\neq\emptyset$, then either $H_Q\subseteq L_z$ or $H_Q\subseteq J_F$. This contradicts the fact that $L_z\cap A_2=\emptyset$ or $J_F\cap A_1=\emptyset$ (see (\ref{eq3}) and (\ref{eq5})). Therefore, $H_Q\cap (L_z\cup J_F)=\emptyset$, for each component $Q$ of $A_1\cap A_2$ such that $p\notin Q$. Let $$U=A_3\cup P'\cup (\cup\{H_Q : Q\text{ is component of }A_1\cup A_2\text{ and }p\notin Q\}).$$ 
By (\ref{eq3}) and (\ref{eq5}), there are $y_z\in L_z\cap (X\setminus (A_2\cup A_3))$ and $y_F\in J_F\cap (X\setminus (A_1\cup A_3))$. Note that $\{y_z,y_F\}\cap U=\emptyset$. 

Let $x\in X\setminus\{y_z,y_F\}$. If $x\in U$, then there is a subcontinuum $Z$ of $X$ such that $\{x\}\cup A_3\subseteq Z$ and $Z\cap\{y_z,y_F\}=\emptyset$. Thus, $\{y_z,y_F\}$ does not block $\{x\}$, by Lemma \ref{Lema5} \textit{(6)}. If $x\notin U$, then $x\in X\setminus (A_2\cup A_3\cup U)$ or $x\in X\setminus (A_1 \cup A_3\cup U)$. Without loss of generality, we may suppose that $x\in X\setminus (A_2\cup A_3\cup U)$. Since $\{y_z\}\in\mathcal{NB}(\mathcal{F}_1(X))$, there is an order arc $\alpha\colon [0,1]\to \mathcal{C}(X)$ such that $\alpha(0)=\{x\}, \alpha(1)=X$ and $y_z\notin\alpha(t)$ for each $t<1$. Let $t_0=\min\{t\in [0,1] : \alpha(t)\cap (A_2\cup A_3)\neq\emptyset\}$. Observe that $t_0<1$, $\alpha(t_0)\subseteq A_1$ and $\alpha(t_0)\cap \Bd_X(A_1)\neq\emptyset$. It is clear that $\alpha(t_0)\cap\{y_z,y_F\}=\emptyset$. Since $\Bd_X(A_1)\subseteq U$, $\alpha(t_0)\cap U\neq\emptyset$. Hence, there is a subcontinuum $T$ of $X$ such that $\alpha(t_0)\cup A_3\subseteq T$ and $T\cap\{y_z,y_F\}=\emptyset$. Thus, $\{y_z,y_F\}$ does not block $\{x\}$, by Lemma \ref{Lema5} \textit{(6)}, and $\{y_z,y_F\}\in\mathcal{NB}(\mathcal{F}_1(X))$. A contradiction. Therefore, $\mathcal{F}_1(X)\subseteq \mathcal{NWC}(X)$ and $X$ is a simple closed curve.
\end{proof}

\vspace{1cm}

\noindent (Javier Camargo and Luis Ortiz)\par
\noindent Escuela de Matem\'aticas, Facultad de Ciencias, Universidad Industrial de
Santander, Ciudad Universitaria, Carrera 27 Calle 9, Bucaramanga,
Santander, A. A. 678, COLOMBIA.\par
\noindent e-mail: jcamargo@saber.uis.edu.co\par
\noindent e-mail: luis0302\_96@hotmail.com\par
\bigskip

\noindent (David Maya)\par
\noindent Facultad de Ciencias, Universidad
Aut\'onoma del Estado de M\'exico,
Instituto Literario 100, Col. Centro, Toluca, CP 50000. M\'EXICO.\par
\noindent e-mail: dmayae@outlook.com, dmayae@uaemex.mx

\end{document}